\documentclass[11pt,a4paper]{amsart}

\usepackage{amsmath}
\usepackage{amsfonts}
\usepackage{graphicx}
\usepackage{amssymb}
\usepackage{xcolor}
\usepackage{etex}
\usepackage{latexsym}
\usepackage[english]{babel}
\usepackage[backref=page]{hyperref}

\def \RN {\mathbb{R}^N}
\def \N {\mathbb{N}}
\def \R {\mathbb{R}}

\def \dsy {\displaystyle}

\def \de {\partial}

\def \LL {\mathcal{L}}
\def \bbX {\mathbb{X}}

\newcommand{\scal}[2]{\langle {#1} , {#2}\rangle}

\usepackage{amsthm}
\theoremstyle{definition}
\newtheorem{definition}{Definition}[section]

\theoremstyle{plain}

\newtheorem{theorem}[definition]{Theorem}
\newtheorem{proposition}[definition]{Proposition}
\newtheorem{lemma}[definition]{Lemma}

\numberwithin{equation}{section}
\begin{document}

 \title[Mixed local and nonlocal operators]{Semilinear elliptic equations \\
involving mixed local and nonlocal operators}
 \date{\today}
 
\author[S.\,Biagi]{Stefano Biagi}
\author[S.\,Dipierro]{Serena Dipierro}
\author[E.\,Valdinoci]{Enrico Valdinoci}
\author[E.\,Vecchi]{Eugenio Vecchi}

\address[S.\,Biagi]{Politecnico di Milano - Dipartimento di Matematica
\newline\indent
Via Bonardi 9, 20133 Milano, Italy}
\email{stefano.biagi@polimi.it}

\address[S.\,Dipierro]{Department of Mathematics and Statistics
\newline\indent University of Western Australia \newline\indent
35 Stirling Highway, WA 6009 Crawley, Australia}
\email{serena.dipierro@uwa.edu.au}

\address[E.\,Valdinoci]{Department of Mathematics and Statistics
\newline\indent University of Western Australia \newline\indent
35 Stirling Highway, WA 6009 Crawley, Australia}
\email{enrico.valdinoci@uwa.edu.au}

\address[E.\,Vecchi]{Politecnico di Milano - Dipartimento di Matematica
\newline\indent
Via Bonardi 9, 20133 Milano, Italy}
\email{eugenio.vecchi@polimi.it}



\thanks{The authors are members of INdAM. S. Biagi
is partially supported by the INdAM-GNAMPA project
\emph{Metodi topologici per problemi al contorno associati a certe
classi di equazioni alle derivate parziali}.
S. Dipierro and E. Valdinoci are members of AustMS and
are supported by the Australian Research Council
Discovery Project DP170104880 NEW ``Nonlocal Equations at Work''.
S. Dipierro is supported by
the Australian Research Council DECRA DE180100957
``PDEs, free boundaries and applications''. E. Vecchi is partially supported
by the INdAM-GNAMPA project
\emph{Convergenze variazionali per funzionali e operatori dipendenti da campi vettoriali}}

 \begin{abstract}
In this paper, we consider an elliptic operator obtained as the superposition
of a classical second-order differential operator and a nonlocal operator of fractional type.
Though the methods that we develop
are quite general, for concreteness we focus on the case in which the
operator takes the form~$-\Delta+(-\Delta)^s$, with~$s\in(0,1)$.
We focus here on symmetry properties of the solutions
and we prove a radial symmetry result, based on the moving plane method,
and a one-dimensional symmetry result,
related to a classical conjecture by G.W. Gibbons.
 \end{abstract}
 
 \maketitle
 
\section{Introduction}

In this article we discuss some symmetry properties for the solutions
of se\-mi\-li\-near equations
driven by a mixed operator.
Specifically, we will consider o\-pe\-ra\-tors that combine local and nonlocal
features. For the sake of concreteness, we focus on operators of the form
\begin{equation}\label{OPER}
\LL: =- \Delta +(-\Delta)^s\end{equation}
where $s\in (0,1)$ and
$$(-\Delta)^s u(x) := \mathrm{P.V.}\int_{\RN}\frac{u(x)-u(y)}{|x-y|^{N+2s}}\,d y.$$
The study of mixed operators has a consolidated
interest in the recent literature, both in terms of
theoretical studies and in view of real-world applications.
The development of the theory includes,
among the others,
viscosity solutions methods (see~\cite{JAK1, JAK2, MR2422079, OHKSBDc3847t8567, CIOM, BBGG-1, BBGG-2}),
parabolic equations (see~\cite{VES}),
Aubry-Mather theory (see~\cite{LL}),
Cahn-Hilliard equations
(see~\cite{6AG}),
porous medium equations (see~\cite{TESO})
phase transitions (see~\cite{CASE}),
fractional damping effects (see~\cite{PATA}),
Bernstein-type regularity results (see~\cite{BERN}),
existence and non-existence results (see~\cite{NIC-MAT, XAVI}),
regularity theory (see~\cite{BDVV}).
Concrete applications of mixed
operators also arise naturally
in plasma physics (see~\cite{PLA})
and population dynamics (see~\cite{EDO}),
and numerical methods
have been also developed to take into account
the specifics of mixed operators (see~\cite{N97564I}).\medskip

In this article, we provide two sets of symmetry results
for solutions
of semilinar equations driven by mixed operators:
the first type of results deals with the radial symmetry of the solutions,
and relies on the moving plane method; the
second type of results is inspired by a classical conjecture by G.W. Gibbons
and establishes the one-dimensional symmetry of the global solutions
that attain uniformly their limit values at infinity.

In this spirit, the first symmetry result that we present is as follows:

\begin{theorem}\label{Simmetria}
Let $f: \mathbb{R} \to \mathbb{R}$ be
a locally Lipschitz continuous function,
and let $\Omega\subset \mathbb{R}^{N}$ be an open and bounded set
with smooth boundary.
We assume that $\Omega$ is symmetric and
convex with respect to the hyperplane $\{ x_1 = 0\}$. 

If $u\in C(\R^N)$ is any weak solution of
\begin{equation}\label{SymmetryProblem}
\begin{cases}
\LL u = f(u) & \textrm{in $\Omega$},\\
u \equiv 0  & \textrm{in $\mathbb{R}^{N} \setminus \Omega$},\\
u> 0  & \textrm{in $\Omega$},
\end{cases}
\end{equation}
then $u$ is symmetric
with respect to $\{x_1 = 0\}$ and increasing in the $x_1$-direction in $\Omega \cap \{x_1 <0\}$.
\end{theorem}
As usual, from Theorem \ref{Simmetria}
one deduces that if~$\Omega$ is a ball, then the solutions of~\eqref{SymmetryProblem}
are necessarily radial and radially decreasing.\medskip

The proof of Theorem~\ref{Simmetria} that we present
combines the integral formulation of the moving plane method
(see \cite{MPS, Sciunzi}) with 
suitable adaptations of some results in \cite{JW}, where the case of
integral equations was taken into account
by introducing
a new small-volume maximum principle and a strong maximum principle for 
an\-ti\-sym\-me\-tric supersolutions.
See also~\cite{339, JAWE, DMPS, MR3713542, MR3827344, MR3937999}
for related moving pla\-ne me\-thods in the nonlocal setting.\medskip

In terms of one-dimensional symmetry for global solutions
under uniform limit assumptions,
we have the following result:
\begin{theorem} \label{thm.Gibbons}
  Let $f\in C^1(\R)$ be such that
  \begin{equation} \label{eq.assumptionf}
   \sup_{|r|\geq 1}f'(r) < 0.
  \end{equation}
  Let $u\in C^3(\R^N)\cap W^{4,\infty}(\R^N)$
  be a solution of the problem
  \begin{equation} \label{eq.GibbonsPb}
   \begin{cases}
  \LL u = f(u) & \text{in $\RN$}, \\
  \dsy\lim_{t\to\pm\infty}u(y,t) = \pm 1 & 
  \text{uniformly for $y\in\R^{N-1}$}.
  \end{cases}
  \end{equation}
  Then,
  there exists $u_0:\R\to\R$ such that
  \begin{equation} \label{eq.onedim}
   u(y,t) = u_0(t) \qquad\text{for every $x = (y,t)\in\RN$}.
  \end{equation}
 \end{theorem}
The result in Theorem~\ref{thm.Gibbons}
is inspired by a classical conjecture
by G.W. Gib\-bons, formulated when~$\LL$ was
the classical Laplace operator
and motivated by the cosmological problem
of detecting the shape of the interfaces which ``separate'' the different regions of
the universe after the big bang
(see~\cite{GIT}).

The classical Gibbons
conjecture was established,
independently and with different methods, by~\cite{BAB, BEB, FAC}.
See also~\cite{689, 468, 579} for related results.

The fractional version of Gibbons conjecture (i.e., the case
in which the o\-pe\-ra\-tor in~\eqref{eq.GibbonsPb} is the fractional
Laplacian)
has been established in~\cite{FarVal, CASI}.
As a matter of fact, the method developed in~\cite{FarVal}
is very general and comprises a number of different
operators in a unified way: for this, our proof of
Theorem~\ref{thm.Gibbons} will rely on the general structure
provided in~\cite{FarVal} by showing that the structural
hypothesis of~\cite{FarVal} are fulfilled in the case
that we consider here.
\medskip

In the rest of the paper we provide the proof of Theorem~\ref{Simmetria},
which is contained in Section~\ref{6-J},
and that of Theorem~\ref{thm.Gibbons}, which is contained
in Section~\ref{sec.main}.

\section{Radial symmetry and proof of Theorem~\ref{Simmetria}}\label{6-J}
In this section, we prove Theorem~\ref{Simmetria}.
To this end, without loss of generality,
we may assume that
$$\inf_{x \in \Omega} x_1 = -1.$$
We will combine the integral version
of the moving plane method (see \cite{MPS}) with a suitable generalization of a 
strong maximum principle for {\it antisymmetric supersolutions}
(see~\cite{JW}). \smallskip

Let us now introduce and fix some notation needed in what follows.
We define the bilinear form
\begin{equation}\label{BILI}
 \begin{split}
  B(u,v):=\;& \int_{\mathbb{R}^{N}} \scal{\nabla u}{\nabla v} \, dx  \\
  &\qquad+ \iint_{\mathbb{R}^{2N}} \dfrac{(u(x)-u(y))(v(x)-v(y))}{|x-y|^{N+2s}}\, dx \, dy,
  \end{split}\end{equation}
 and the function space
	\begin{equation}\label{BILI2}
	\mathcal{D}(\Omega):= \left\{ u\in H^{1}(\mathbb{R}^{N}) \, \textrm{ s.t. } 
	u \equiv 0 \textrm{ in } \mathbb{R}^{N} \setminus \Omega \right\}.
	\end{equation}
In this setting, we give the following definition of weak solution
of~\eqref{SymmetryProblem}:
\begin{definition}\label{def:uweak}
We say that a function $u:\Omega\to\R$ is a weak solution of~\eqref{SymmetryProblem}
if~$u\in \mathcal{D}(\Omega)$, $u>0$ in~$\Omega$, and
\begin{equation}\label{ADMI}
B(u,\varphi)=\int_{\R^N} f(u(x))\varphi(x)\,dx,\end{equation}
for any~$\varphi\in \mathcal{D}(\Omega)$.
\end{definition}
Also, given a set $U \subset \mathbb{R}^{N}$,
we let
 \begin{equation}\label{FURHO}
  \rho(v,U) := \int_{U}|\nabla v|^2 + [v]^{2}_{H^{s}(U)},
 \end{equation}
where
$$ [v]^{2}_{H^{s}(U)}:=\iint_{U \times U} 
\dfrac{|v(x)-v(y)|^2}{|x-y|^{N+2s}}\,dx\,dy,$$
and
\begin{equation}\label{HUDEF}
\mathcal{H}(U):= \left\{ v\in L^{2}(\mathbb{R}^{N})\,\,\text{s.t.\,$v\in H^{1}(U)$}\right\}.
\end{equation}
As customary, for any $v \in L^{2}(\mathbb{R}^{N})$ we define the
positive and negative parts of~$v$ as follows
\begin{equation*}
 v^{+} := \max \{ v, 0\}\qquad{\mbox{and}}\qquad
 v^{-}:= \max\{-v,0 \}.
\end{equation*}
As it is well known, 
\begin{equation}\label{fact1}
v(x) = v^{+}(x) - v^{-}(x), \quad \textrm{for a.e. } x \in \mathbb{R}^{N}
\end{equation}
\noindent and 
\begin{equation}\label{fact2}
 v^{+}(x)v^{-}(x) = 0, \quad \textrm{for a.e. } x \in \mathbb{R}^{N}.
\end{equation}
It is useful to observe that the functional introduced in~\eqref{FURHO}
is monotone with respect to the operation of taking the positive
and negative parts, as pointed out in the following result:
\begin{lemma}\label{positivepart}
 Let $U \subset \mathbb{R}^{N}$ be an open set and let~$v \in \mathcal{H}(U)$.
 Then~$v^{\pm} \in \mathcal{H}(U)$ and
 \begin{equation}\label{eq:Lemmapositive}
  \rho(v^{\pm},U) \leq \rho(v,U).
  \end{equation}
\begin{proof}
Since~$v\in\mathcal{H}(U) = L^2(\RN)\cap H^1(U)$, it is easy to see
that~$v^{\pm} \in \mathcal{H}(U)$, in light of~\eqref{fact1}
and~\eqref{fact2}. We then focus on the proof of~\eqref{eq:Lemmapositive}.

For this, recalling \eqref{FURHO} and using again 
\eqref{fact1}-\eqref{fact2}, we get
\begin{equation*}
 \begin{split}
 \rho(v,U)&\,= \int_{U}|\nabla v|^2 + \iint_{U \times U} 
 \dfrac{|v(x)-v(y)|^2}{|x-y|^{N+2s}}\,dx\,dy \\
 &\,=\int_{U}|\nabla (v^{+}- v^{-})|^2 + \iint_{U \times U} 
 \dfrac{|(v^{+}- v^{-})(x)-(v^{+}- v^{-})(y)|^2}{|x-y|^{N+2s}}\,dx\,dy\\
 &\,= \int_{U}|\nabla v^{+}|^2 +
 \int_{U}|\nabla v^{-}|^2\\&\qquad
 +
  \iint_{U \times U} 
 \dfrac{|v^{+}(x)-v^{+}(y)|^2}{|x-y|^{N+2s}}\,dx\,dy+
 \iint_{U \times U} 
 \dfrac{| v^{-}(x)-v^{-}(y)|^2}{|x-y|^{N+2s}}\,dx\,dy\\&\qquad-2
 \iint_{U \times U} 
 \dfrac{\big(v^{+}(x)- v^{+}(y)\big)\big(v^{-}(x)- v^{-}(y)\big)}{
 |x-y|^{N+2s}}\,dx\,dy\\&\,=
 \rho(v^{+},U)+\rho(v^{-},U)+2\iint_{U \times U} 
 \dfrac{v^{+}(x)v^{-}(y)+ v^{+}(y)v^{-}(x)}{
 |x-y|^{N+2s}}\,dx\,dy\\&\,\ge
 \rho(v^{+},U)+\rho(v^{-},U),
 \end{split}
\end{equation*}
which gives the desired
result in~\eqref{eq:Lemmapositive}.
	\end{proof}
	\end{lemma}
Inspired by~\cite{JW}, we now deal with a linear problem associated to
the re\-flec\-tion with respect to a given hyperplane.
For this, 
with the notation in~\eqref{BILI} and~\eqref{BILI2},
for every open and bounded set~$\Omega \subset \mathbb{R}^{N}$, we define the
first (variational) ei\-gen\-va\-lue of the operator~$\LL$ introduced in~\eqref{OPER}
as 
	\begin{equation}\label{eq:Lambda1}
	\Lambda_1 (\Omega) := \inf_{u \in \mathcal{D}(\Omega)} \dfrac{B(u,u)}{\|u\|^2_{L^{2}(\Omega)}}.
	\end{equation}

We point out that, since we can identify $\mathcal{D}(\Omega)$
with the space of functions in~$H^{1}_{0}(\Omega)$ 
that vanish outside~$\Omega$, we see that
\begin{equation}\label{laporhgn}
\Lambda_{1}(\Omega) \geq \Lambda_{-\Delta}(\Omega),\end{equation}
where $\Lambda_{-\Delta}(\Omega)$ stands for the first
eigenvalue of~$-\Delta$ in~$\Omega$ with homogeneous Dirichlet
boundary conditions.
Recalling that
$$
\Lambda_{-\Delta}(\Omega) \to +\infty \quad \textrm{ as } |\Omega| \to 0,
$$
and setting 
	\begin{equation*}
	\Lambda_1 (r) := \inf \big\{ \Lambda_1 (\Omega) {\mbox{ with }}
\Omega \subset \mathbb{R}^n \textrm{ open  with } |\Omega|=r \big\}, \quad r >0, 
	\end{equation*}
	it follows from~\eqref{laporhgn}
that
\begin{equation}\label{LemmaAuto}
\Lambda_1 (r) \to +\infty\quad {\mbox{ as }}r \to 0^{+}.
	\end{equation}
Furthermore, let~$H \subset \mathbb{R}^{N}$ be an open and
affine halfspace. We denote by
$$Q:\mathbb{R}^{N} \to \mathbb{R}^{N}$$ 
the reflection with respect to $\partial H$. 
For convenience, we will sometimes denote with
\begin{equation}\label{RECALL}\bar{x} := Q(x),\end{equation}
for every $x \in \mathbb{R}^{N}$.
With this notation at hand, we say that a function $v:\mathbb{R}^{N} \to \mathbb{R}$
is {\it antisymmetric with respect to $Q$} if 
 \begin{equation}\label{ANTIS}
  v(\bar{x}) = - v(x), \quad \textrm{for every } x \in \mathbb{R}^{N}.
  \end{equation}
Moreover, we give the following definition of antisymmetric supersolutions:
\begin{definition}\label{Soprasol}
 Let $U \subset H$ be an open and bounded set.
 Let~$c \in L^{\infty}(U)$. We say that a
 function $v: \mathbb{R}^{N} \to \mathbb{R}$ is an antisymmetric supersolution of
 \begin{equation}\label{eq:AntisymmProblem}
 \begin{cases}
  \LL v = c v & \textrm{in } U,\\
	v \equiv 0& \textrm{in } H \setminus U,
 \end{cases}
 \end{equation}
 if it satisfies the following properties:
 \begin{itemize}
  \item[(i)] $v$ is antisymmetric,
  \item[(ii)] $v \in \mathcal{H}(U')$ for some open
 set~$U' \subset \mathbb{R}^{N}$ such that $Q(U')=U'$ and~$\overline{U} \subset U'$,
 \item[(iii)] $ v \geq 0$ in $H \setminus U$ and,
 for every $\varphi \in \mathcal{D}(U)$ with $\varphi \geq 0$, one has
 \begin{equation}\label{TEST}
  B(v,\varphi) \geq \int_{U} c(x)v(x)\varphi(x) \, dx.
  \end{equation}
 \end{itemize}
 \end{definition}
The aim is now to provide a suitable maximum principle for antisymmetric
supersolutions, as given in Definition~\ref{Soprasol}.	 \medskip

We start with the following observation on the bilinear form introduced
in~\eqref{BILI}:
 \begin{lemma}\label{wammissibile}
  Let $U' \subset \mathbb{R}^{N}$ be an open set such that $Q(U')=U'$.
  Let $v \in \mathcal{H}(U')$
  be an antisymmetric function such that
 \begin{equation}\label{agg1}
  v \geq 0\quad {\mbox{ in }} H \setminus U,\end{equation}
  for a certain open and bounded set~$U \subset H$ with the property that
 \begin{equation}\label{agg2}
  \overline{U} \subset H\cap U'.
 \end{equation}
 Then, the function
 \begin{equation}\label{DEFW}
 w:= \chi_{H} v^{-} \in \mathcal{D}(U)
 \end{equation} and it holds that
 \begin{equation}\label{eq:BwwBvw}
	B(w,w) \leq - B(v,w).
	\end{equation}
\begin{proof}
We first prove~\eqref{DEFW}.
To this end we first observe that, 
since~$v \in L^{2}(\mathbb{R}^{N})$,
one obviously has $w\in L^{2}(\mathbb{R}^{N})$.
Also, recalling~\eqref{HUDEF},
we know that~$v \in H^{1}(U')$, and therefore it is easy to see
that~$v^-\in H^1(U')$.
In addition, in light of~\eqref{agg1}, we have that~$v^{-}\equiv0$
in~$H\setminus U$. 
As a consequence of these observations and of~\eqref{agg2}, we have that
there exists an open set~$W$
such that
$$\text{$\overline{U}\subset W\subset\overline{W}\subset U'\cap H$
and~$v^-\in H_0^1(W)$.}$$
Therefore, if we identify $w = \chi_{H}v^{-}$ with the zero extension of $v^{-}$ outside of $U$,
we get that~$w \in H^{1}(\mathbb{R}^{N})$.
Moreover, we have that~$w\equiv0$ in~$\R^N\setminus U$. 
These considerations imply~\eqref{DEFW}.

Now we 
focus on the proof of~\eqref{eq:BwwBvw}.
Recalling \eqref{BILI}, we observe that
\begin{equation}\label{dihfugvbvjnb0}
 \begin{split}
  &	B(w,w) + B(v,w) \\=\;&   \int_{\mathbb{R}^{N}}|\nabla w|^2 \, dx 
  + \iint_{\mathbb{R}^{2N}} \dfrac{(w(x)-w(y))^2}{|x-y|^{N+2s}}\, dx \, dy\\&\qquad+
  \int_{\mathbb{R}^{N}} \scal{\nabla v}{\nabla w} \, dx 
  +  \iint_{\mathbb{R}^{2N}} \dfrac{(v(x)-v(y))(w(x)-w(y))}{|x-y|^{N+2s}}\, dx \, dy.
\end{split}	
\end{equation}
We notice that, thanks to~\eqref{agg1},
\begin{equation}\label{dihfugvbvjnb}
\begin{split}
\int_{\mathbb{R}^{N}} 
|\nabla w|^2 \, dx & + \int_{\mathbb{R}^{N}} \scal{\nabla v}{\nabla w} \, dx 
= \int_{U} 
|\nabla v^{-}|^2 \, dx + \int_{U} \scal{\nabla v}{\nabla v^-} \, dx \\[0.1cm]
& = \int_{U}|\nabla v^{-}|^2 \, dx -\int_{U} \scal{\nabla v^{-}}{\nabla v^-} \, dx
= 0.
\end{split}
\end{equation}
Furthermore, we remark that, for any~$x\in\mathbb{R}^{N}$,
\begin{align*}
& w(x)\big(w(x)+v(x)\big)\\
&\qquad = \chi_H(x)v^{-}(x)\big(\chi_H(x)v^{-}(x)+
\chi_H(x)v(x)+\chi_{\R^N\setminus H}(x)v(x)\big)\\
& \qquad = \chi_H(x)v^{-}(x)\big(\chi_H(x)v^{+}(x)+\chi_{\R^N\setminus H}(x)v(x)\big) = 0,
\end{align*}
and therefore
\begin{align*}
& (w(x)-w(y))^2+(v(x)-v(y))(w(x)-w(y))
\\
&\qquad = (w(x)-w(y))\big((w(x)+v(x))-(w(y)+v(y))\big)\\
& \qquad = -w(x)(w(y)+v(y))-w(y)(w(x)+v(x)).
\end{align*}
As a consequence, using~\eqref{ANTIS} and the change of
variable~$Y:=\bar y$ (also recall the notation in~\eqref{RECALL}),
we obtain
\begin{align*}
 & \iint_{\mathbb{R}^{2N}} \dfrac{(w(x)-w(y))^2}{|x-y|^{N+2s}}\, dx \, dy
 + \iint_{\mathbb{R}^{2N}} \dfrac{(v(x)-v(y))(w(x)-w(y))}{|x-y|^{N+2s}}\, dx \, dy \\[0.1cm]
 &\quad = 
  \iint_{\mathbb{R}^{2N}} \dfrac{w(x)(w(y)+v(y))+w(y)(w(x)+v(x))}{|x-y|^{N+2s}}\, dx \, dy \\[0.1cm]
  & \quad = -2\iint_{\mathbb{R}^{2N}} \dfrac{w(x)(w(y)+v(y))}{|x-y|^{N+2s}}\, dx \, dy \\[0.1cm]
  & \quad = -2\iint_{H\times\mathbb{R}^{N}} 
    \dfrac{v^{-}(x)\big(\chi_H(y)v^{-}(y)+v(y)\big)}{|x-y|^{N+2s}}\, dx \, dy \\[0.1cm]
  & \quad = -2\iint_{H\times\mathbb{R}^{N}} 
  \dfrac{v^{-}(x)\big(\chi_H(y)v^{+}(y)+\chi_{\R^N\setminus H}(y)v(y)\big)}
  {|x-y|^{N+2s}}\,dx\,dy \\[0.1cm]
  & \quad = -2\iint_{H\times H} \dfrac{v^{-}(x)v^{+}(y)}{|x-y|^{N+2s}}\,dx\,dy
  -2\iint_{H\times(\R^N\setminus H)} \dfrac{v^{-}(x)v(y)}{|x-y|^{N+2s}}\, dx \, dy \\[0.1cm]
  & \quad = -2\iint_{H\times H} \dfrac{v^{-}(x)v^{+}(y)}{|x-y|^{N+2s}}\,dx\,dy 
  + 2\iint_{H\times (\R^N\setminus H)} \dfrac{v^{-}(x)v( \bar y)}{|x- y|^{N+2s}}\,dx\,d y \\[0.1cm]
  & \quad = -2\iint_{H\times H} \dfrac{v^{-}(x)v^{+}(y)}{|x-y|^{N+2s}}\,dx\,dy
  + 2\iint_{H\times H} \dfrac{v^{-}(x)v( Y)}{|x- \bar Y|^{N+2s}}\,dx\,dY \\[0.1cm]
  & \quad = -2\iint_{H\times H} v^{-}(x)v^{+}(y)\left(\frac1{|x-y|^{N+2s}} 
  - \frac{1}{|x- \bar y|^{N+2s}}\right)\,dx\,dy \\[0.1cm]
  & \qquad\qquad -2\iint_{H\times H} \dfrac{v^{-}(x)v^{-}( y)}{|x- \bar y|^{N+2s}}\, dx \, dy
   \\
   & \quad \leq 0.
\end{align*}
Plugging this information and~\eqref{dihfugvbvjnb}
into~\eqref{dihfugvbvjnb0} we obtain~\eqref{eq:BwwBvw},
as desired.
\end{proof}
\end{lemma}
With the aid of Lemma~\ref{wammissibile}, we now prove
the following maximum principle:	
\begin{proposition}\label{PreSMP}
Let $U \subset \mathbb{R}^{N}$ be an open and bounded set
with~$\overline{U} \subset H$.
Mo\-re\-o\-ver, let~$c \in L^{\infty}(U)$ be such that 
\begin{equation}\label{068yhgnfdjfdk}
\|c^{+}\|_{L^{\infty}(U)} < \Lambda_{1}(U),\end{equation}
where the notation in~\eqref{eq:Lambda1} has been used.

Then, every antisymmetric supersolution $v$ of \eqref{eq:AntisymmProblem} 
\emph{(}in $U$\emph{)}
is nonnegative throughout $H$, that is, $v(x) \geq 0$ for a.e.\,$x \in H$.
 \begin{proof}
 We consider the function~$w $ introduced in~\eqref{DEFW}
 and we claim that
 \begin{equation}\label{ruthrughri}
  w\equiv0.
 \end{equation}
To prove it, we argue towards a contradiction, supposing
that~$\|w\|_{L^{2}(U)} \neq 0$.
By Lemma \ref{wammissibile}, we know that~$w \in \mathcal{D}(U)$, and hence it is an admissible test function in~\eqref{TEST}.
Accordingly,
$$ B(v,w)\ge \int_{U}c(x)v(x)w(x)\, dx.$$
{F}rom this, \eqref{eq:Lambda1}, \eqref{eq:BwwBvw}
and~\eqref{068yhgnfdjfdk}, 
we conclude that
\begin{equation*}
 \begin{aligned}
 \Lambda_{1}(U) \|w\|_{L^{2}(U)}^2 &\leq B(w,w) \leq -B(v,w) \leq - \int_{U}c(x)v(x)w(x)\, dx\\
 &= \int_{U}c(x)w^2(x) \, dx \leq \|c^{+}\|_{L^{\infty}(U)}\|w\|_{L^{2}(U)}^2 
 < \Lambda_{1}(U) \|w\|_{L^{2}(U)}^2,
 \end{aligned}
\end{equation*}
which is a contradiction. This proves~\eqref{ruthrughri},
we implies the desired result.
\end{proof}
\end{proposition}
 We are now in the position of establishing a strong
 maximum principle for antisymmetric supersolutions
 (which is the counterpart in the setting of mixed local-nonlocal
 operators of~\cite[Proposition 3.6]{JW}): 
\begin{proposition}\label{StrongMax}
 Let $U \subset H$ be an open and bounded set. Let $c \in L^{\infty}(U)$ and let $v$ be 
 an antisymmetric supersolution of \eqref{eq:AntisymmProblem}
 \emph{(}in~$U$\emph{)}. 
 Assume that 
 \begin{equation}\label{ASPRE}
 v \geq 0\quad \text{a.e.\,in $H$}.
 \end{equation}
 Then, either $v \equiv 0$ in $\mathbb{R}^N$ or
 \begin{equation*}
 \mathrm{ess \, inf}_{K} v >0, \quad \textrm{for every compact set } K \subset U.
 \end{equation*}
 \begin{proof}
 If $v \equiv 0$ in $\mathbb{R}^N$, there is nothing to prove,
 so we assume that
 \begin{equation}\label{assume}
 v\not\equiv 0\quad \text{in $\mathbb{R}^N$}.
 \end{equation}
 In this case, it suffices to show
 that, for a fixed $x_0 \in U$, one has
\begin{equation}\label{jr984thgjfbvjfbv}
\mathrm{ess \, inf}_{B_r(x_0)} v >0,
\end{equation}
for a some radius~$r>0$ small enough.
We then prove \eqref{jr984thgjfbvjfbv}. \vspace*{0.05cm}

First of all, in light of~\eqref{ASPRE},
\eqref{assume} and the fact that~$v$ is antisymmetric, we can find a bounded set $M\subset H$,
with positive measure, which
does not contain a small neighborhood of~$x_0$ and such that
\begin{equation}\label{defDelta}
\delta:= \inf_M v >0.
\end{equation} 
In addition, by~\eqref{LemmaAuto}, we find a radius 
\begin{equation}\label{doewghbn}
r \in \left( 0, \frac{\mathrm{dist}(x_0; (\mathbb{R}^N \setminus H)\cup M)}{4}\right) \end{equation}
such that 
\begin{equation}\label{risolto}
\Lambda_1 (B_{2r}(x_0))> \|c\|_{L^{\infty}(U)}.
\end{equation}
 We now pick a function $ g\in C^2_{0}(\mathbb{R}^N, [0,1])$
 such that
 \begin{equation*}
  g(x):= \begin{cases}
  1, & {\mbox{ if }} x\in B_r(x_0),\\
  0, & {\mbox{ if }} x\in\R^N\setminus B_{2r}(x_0).
 \end{cases}
 \end{equation*}
Moreover, for a given $a>0$ to be chosen later, we define the
function
\begin{equation}\label{defahha}
h:\RN\to\R, \quad h(x):= g(x) - g(\bar{x}) + a \left( \chi_M (x) - \chi_M (\bar{x})\right),
\end{equation}	 
where we are using the notation in~\eqref{RECALL}.
We also define the sets $U_0 := B_{2r}(x_0)$ and
$U'_0 := B_{3r}(x_0) \cup Q(B_{3r}(x_0))$.

We observe that~$h$ is antisymmetric, and moreover
\begin{equation}\label{USEF}
h\equiv 0 \; \textrm{ on } H \setminus (U_0 \cup M) \quad{\mbox{ and }}
\quad h \equiv a\; \textrm{ on } M,
\end{equation}
thanks to~\eqref{doewghbn}.
{F}rom~\eqref{doewghbn} we also deduce that
\begin{equation}\label{EmptyIntersection}
\left( M  \cup Q(M) \right) \cap U'_0 = \varnothing.
\end{equation}
This and the fact that~$M$ is bounded give that~$h \in \mathcal{H}(U'_{0})$.
We now claim that there exists a constant 
$C_1>0$, depending on~$g$, such that
\begin{equation}\label{oejgirgij00}
B(g,\varphi) \leq C_1\int_{U_0}\varphi(x)\,dx, \quad \textrm{for every } \varphi \in \mathcal{D}(U_0) {\mbox{ with }}\varphi\ge0.
\end{equation}
In fact, for any~$\varphi \in \mathcal{D}(U_0)$ with~$\varphi\ge0$, by
an integration by parts,
\begin{equation}\begin{split}\label{oejgirgij}
& \int_{\mathbb{R}^{N}} \scal{\nabla g}{\nabla \varphi} \, dx =
\int_{U_0} \scal{\nabla g}{\nabla \varphi} \, dx\\&\qquad
=-\int_{U_0} \Delta g\,\varphi \, dx\le
\|g\|_{C^2(\R^N)}\int_{U_0}\varphi(x)\,dx.\end{split}
\end{equation}
Moreover, by Proposition~2.3-(ii) in~\cite{JW} (applied here
with~$v:=g$ and~$u:=\varphi$), we have that
\begin{eqnarray*}
&&\frac{1}{2} \iint_{\mathbb{R}^{2N}} \dfrac{(g(x)-g(y))(\varphi(x)-\varphi
 (y))}{|x-y|^{N+2s}}\, dx \, dy
 =\int_{\R^N}(-\Delta)^s g(x)\,\varphi(x)\,dx\\&&\qquad
 =\int_{U_0}(-\Delta)^s g(x)\,\varphi(x)\,dx\le
 \|(-\Delta)^s g\|_{L^\infty(U_0)}\int_{U_0}\varphi(x)\,dx.
\end{eqnarray*}
Recalling~\eqref{BILI},
this and~\eqref{oejgirgij} imply~\eqref{oejgirgij00}.
Similarly, one has that
\begin{equation}\label{oejgirgij0022}
B(g\circ Q,\varphi) \leq C_2 \int_{U_0}\varphi(x)\,dx, \quad \textrm{for every } \varphi \in \mathcal{D}(U_0)
 {\mbox{ with }}\varphi\ge0,
\end{equation}
for some~$C_2>0$.
In addition, we see that, for any~$\varphi \in \mathcal{D}(U_0)$
and any~$x\in\R^N$, from \eqref{doewghbn} we infer that
\begin{eqnarray*}
(\chi_M (x) - \chi_M (\bar{x}))\varphi(x)=0;
\end{eqnarray*}
as a consequence,
\begin{equation}\begin{split}\label{49bvrgk}
 &\frac12\iint_{\mathbb{R}^{2N}} \dfrac{\big(
 (\chi_M (x) - \chi_M (\bar{x}))-(\chi_M (y) - \chi_M (\bar{y}))\big)
 (\varphi(x)-\varphi(y))}{|x-y|^{N+2s}}\,dx\,dy  \\
 &\qquad=-\frac12\iint_{\mathbb{R}^{2N}} \dfrac{(\chi_M (x) - \chi_M (\bar{x}))\varphi(y)+
 (\chi_M (y) - \chi_M (\bar{y}))\varphi(x)}{|x-y|^{N+2s}}\,dx\,dy \\
  &\qquad= -\iint_{U_0\times\R^N} \dfrac{(\chi_M (y) - \chi_M (\bar{y}))\varphi(x)}
  {|x-y|^{N+2s}}\, dx \, dy \\&
  \qquad=-\iint_{U_0\times\mathbb{R}^{N}}
  \dfrac{(\chi_M (y) - \chi_M (\bar{y}))\varphi(x)}{|x-y|^{N+2s}}\,dx\,dy\\
  &\qquad=-\int_{U_0}\varphi(x)\left(\int_M \dfrac{dy}{|x-y|^{N+2s}}
  -\int_{Q(M)} \dfrac{dy}{|x-y|^{N+2s}}\right)\,dx\\&\qquad=
  - \int_{U_0}\varphi(x)\left(\int_M \dfrac{dy}{|x-y|^{N+2s}}
  -\int_{M} \dfrac{dy}{|x-\bar y|^{N+2s}}\right)\,dx \\
  &\qquad\le-C_0\int_{U_0}\varphi(x)\,dx,
\end{split}\end{equation}
where
$$ C_0:=\inf_{x\in U_0}\left(\int_M \dfrac{dy}{|x-y|^{N+2s}}
-\int_{M} \dfrac{dy}{|x-\bar y|^{N+2s}}\right) .
$$
We stress on the fact that the constant~$C_0$ is finite, thanks to~\eqref{doewghbn}.

Now, recalling~\eqref{defahha}, and using~\eqref{oejgirgij00},
\eqref{oejgirgij0022} and~\eqref{49bvrgk}, we conclude that,
for any~$\varphi \in \mathcal{D}(U_0)$, one has
\begin{equation}\begin{split}\label{jirhgierhg}
& B(h,\varphi) =
B(g,\varphi) + B(g\circ Q,\varphi)
\\[0.1cm]
&\qquad +{a} \iint_{\mathbb{R}^{2N}} \dfrac{\big((\chi_M (x) - \chi_M (\bar{x}))-
(\chi_M (y) - \chi_M (\bar{y}))\big)(\varphi(x)-\varphi(y))}{|x-y|^{N+2s}}\,dx\,dy\\[0.1cm]
& \,\,\le
C_a\int_{U_0}\varphi(x)\,dx,
\end{split}
\end{equation}
where
$$C_a:=C_1+C_2-2a\,C_0.$$
Now we perform our choice of the parameter~$a$: we choose~$a>0$ such that
$$C_a < -\|c\|_{L^{\infty}(U_0)}.$$
In particular, with this choice, \eqref{jirhgierhg} yields that
\begin{equation}\label{jfiehgvdjbv497648}\begin{split}
&B(h,\varphi) \leq -\|c\|_{L^{\infty}(U_0)} \int_{U_0}\varphi(x)\,dx
 \leq-\|c^{-}\|_{L^{\infty}(U_0)} \int_{U_0} \varphi(x)\,dx\\&\qquad\le
- \int_{U_0} c^-(x)\varphi(x)\,dx\le 
- \int_{U_0} c^-(x)h(x)\varphi(x)\,dx\\&\qquad
\le  \int_{U_0} c^+(x)h(x)\varphi(x)\,dx
- \int_{U_0} c^-(x)h(x)\varphi(x)\,dx\\&\qquad=
\int_{U_0} c(x)h(x)\varphi(x)\,dx
,\end{split}
\end{equation}
since~$h(x)=g(x)\in[0,1]$ for every~$x\in U_0$.
Now, we recall~\eqref{defDelta}, we define the function $\tilde{v}$ as
\begin{equation}\label{defvtilde}
\tilde{v}(x) := v(x) - \dfrac{\delta}{a}h(x),\end{equation}
and we notice that~$\tilde{v} \in \mathcal{H}(U_0')$ and it
is antisymmetric, since both~$v$ and~$h$ are so.
Furthermore, by~\eqref{ASPRE}, \eqref{defDelta} and~\eqref{USEF}, we have that
$$ \tilde{v}\geq 0\quad {\mbox{on }}H\setminus U_0.$$
In addition, for any~$\varphi\in \mathcal{D}(U_0)$ with~$\varphi\ge0$,
\begin{align*}
B(\tilde v,\varphi) & = B( v,\varphi)-\dfrac{\delta}{a}B(h,\varphi)\\
&\ge \int_{U_0}c(x)v(x)\varphi(x)\,dx
-\dfrac{\delta}{a}\int_{U_0}c(x)h(x)\varphi(x)\,dx \\
& = 
\int_{U_0}c(x)\tilde v(x)\varphi(x)\,dx,
\end{align*}
thanks to~\eqref{TEST}
and~\eqref{jfiehgvdjbv497648}.

As a consequence, we have that~$\tilde{v}$ is an antisymmetric
supersolution of 
\begin{equation*}
\begin{cases}
\mathcal{L} \tilde{v} =c \tilde{v} & \text{in $U_0$},\\
 \tilde{v}\equiv 0 & \text{in $H \setminus U_0$}.
\end{cases}
\end{equation*}
Since $\|c\|_{L^{\infty}(U_0)}<\Lambda_1 (U_0)$, thanks to~\eqref{risolto},
we are in the position to apply Pro\-po\-si\-tion~\ref{PreSMP} to conclude that $\tilde{v} \geq 0$ a.e.
on $U_0$. Recalling~\eqref{defvtilde},
this gives 
$$v \geq \frac{\delta}{a}>0 \quad{\mbox{ a.e. on }}B_r(x_0).$$
This establishes~\eqref{jr984thgjfbvjfbv}, and the proof of
Proposition~\ref{StrongMax}
is thereby complete.
\end{proof}
\end{proposition}
With this preliminary work, we now prove Theorem~\ref{Simmetria}.
For this, let~$u \in C(\Omega) $ be a weak solution
of~\eqref{SymmetryProblem}.
We fix the usual notation needed to implement the moving plane method.
For every $\lambda \in (-1,0]$ we define the following:
\begin{align*}
\Omega_{\lambda} & := \{x \in \Omega: x_1 < \lambda\},\\
\Sigma_{\lambda} &:=  \{x \in \mathbb{R}^N: x_1 < \lambda\},\\
Q_{\lambda}(x) & = x_{\lambda}:= (2\lambda -x_1, x_2, \ldots, x_N),\\
{\mbox{and }}\quad u_{\lambda}(x) &:= u(x_{\lambda}).
\end{align*}
We also define the function
\begin{equation}\label{eq:defc}
 c(x):= \begin{cases}
  \displaystyle  \dfrac{f(u_{\lambda}(x) )- f(u(x)))}{u_{\lambda}(x)-u(x)}, & 
  \text{if $u_{\lambda}(x)\neq u(x)$},\\
  0, & \textrm{if $u_{\lambda}(x)= u(x)$}.
 \end{cases}
\end{equation}
We observe that~$c\in L^{\infty}(\Omega_{\lambda})$, thanks to the
Lipschitz assumption on~$f$.

Furthermore, setting
\begin{equation}\label{degjidj}
v_\lambda := u_{\lambda} - u,
\end{equation}
we point out 
the following observation:
\begin{lemma}\label{lema:weak}
Let~$u$ be a weak solution of~\eqref{SymmetryProblem}
according to Definition~\ref{def:uweak}.

Then, the function~$v_\lambda$ in~\eqref{degjidj} is an
antisymmetric supersolution of~\eqref{eq:AntisymmProblem} in
$\Omega_\lambda$,
according to Definition~\ref{Soprasol},
with~$c$ as in~\eqref{eq:defc}.
\end{lemma}
\begin{proof}
We notice that~$v_\lambda\in H^1(\R^N)\subset \mathcal{H}(U')$, 
for every open set~$U' \subset \mathbb{R}^{N}$
such that~$Q(U')=U'$ and~$\overline{\Omega_\lambda} \subset U'$.
Moreover, since~$u \geq 0$ in~$\R^N$ and~$u \equiv 0$ on~$\Sigma_\lambda \setminus \Omega_\lambda$,
we have that~$v_\lambda \geq 0$ on~$\Sigma_\lambda \setminus \Omega_\lambda$.
In addition, for any~$\varphi\in \mathcal{D}(\Omega_\lambda)$
and for any~$x\in\R^N$, we have
\begin{equation}\begin{split}\label{dwejgerheruhebhrnbriwt45846}
&\scal{\nabla u_\lambda(x)}{\nabla \varphi(x)}
=\big(-\partial_1 u,\partial_2 u, \cdots,\partial_N u\big)(\bar x)\cdot
\big(\partial_1 \varphi,\partial_2 \varphi, \cdots,\partial_N \varphi\big)(x)
\\&\qquad =
\big(\partial_1 u,\partial_2 u, \cdots,\partial_N u\big)(X)\cdot
\big(-\partial_1 \varphi,\partial_2 \varphi, \cdots,\partial_N \varphi\big)(\bar X)\\
&\qquad =\scal{\nabla u(X)}{\nabla \varphi_\lambda(X)},
\end{split}\end{equation}
where~$X:=\bar x$.
Similarly, setting also~$Y:=\bar y$,
\begin{eqnarray*}
&&  \dfrac{(u_\lambda( x)-u_\lambda( y))(
\varphi(x)-\varphi(y))}{|x-y|^{N+2s}} =
 \dfrac{(u(\bar x)-u(\bar y))(
\varphi(x)-\varphi(y))}{|x-y|^{N+2s}}\\&&\qquad=
\dfrac{(u(X)-u(Y))(
\varphi(\bar X)-\varphi(\bar Y))}{|\bar X-\bar Y|^{N+2s}}=
\dfrac{(u(X)-u(Y))(
\varphi(\bar X)-\varphi(\bar Y))}{|X- Y|^{N+2s}}\\&&\qquad=
\dfrac{(u(X)-u(Y))(
\varphi_\lambda( X)-\varphi_\lambda(Y))}{| X- Y|^{N+2s}}.
\end{eqnarray*}
{F}rom this and~\eqref{dwejgerheruhebhrnbriwt45846}, we obtain that
\begin{align*}
& B(u_{\lambda},\varphi) \\
& \quad = \int_{\mathbb{R}^{N}} \scal{\nabla u_\lambda(x)}{\nabla \varphi(x)} \,dx 
+ \iint_{\mathbb{R}^{2N}} \dfrac{(u_\lambda( x)-u_\lambda( y))(
\varphi(x)-\varphi(y))}{|x-y|^{N+2s}}\,dx\,dy \\
& \quad = 
\int_{\mathbb{R}^{N}} \scal{\nabla u(X)}{\nabla \varphi_\lambda
( X)} \, dX 
+  \iint_{\mathbb{R}^{2N}} \dfrac{(u(X)-u(Y))(\varphi_\lambda( X)-\varphi_\lambda(Y))}
{|X- Y|^{N+2s}}\,dX\,dY \\
&\quad = B(u,\varphi_\lambda).
\end{align*}
As a consequence, since~$\varphi_\lambda\in \mathcal{D}(Q_\lambda(
\Omega_\lambda))\subset \mathcal{D}(\Omega)$,
we can use Definition~\ref{def:uweak} to find that
\begin{align*}
B(u_{\lambda},\varphi)& = \int_{\R^N} f(u(x))\varphi_\lambda(x)\,dx\\
& = \int_{\R^N} f(u(\bar X))\varphi(X)\,dX=
\int_{\R^N} f(u_\lambda(X))\varphi(X)\,dX .
\end{align*}
Therefore,
\begin{align*}
B(v_\lambda,\varphi)& =B(u_{\lambda},\varphi)-B( u,\varphi)\\
& =
\int_{\R^N} f(u_\lambda(x))\varphi(x)\,dx-
\int_{\R^N} f(u(x))\varphi(x)\,dx\\
& = \int_{\R^N}c(x)v_\lambda(x)\varphi(x)\,dx,
\end{align*}
which proves~\eqref{TEST}, and thereby completes the proof
of Lemma~\ref{lema:weak}.
\end{proof}
With these considerations, we are now ready to prove Theorem \ref{Simmetria}.
\begin{proof}[Proof of Theorem \ref{Simmetria}]
For every~$\lambda \in (-1,0)$, we define the
function
\begin{equation}\label{eq:wlambda}
 w_{\lambda}:\mathbb{R}^{N} \to \mathbb{R}, \quad
 w_{\lambda}(x):= \begin{cases}
 (u-u_{\lambda})^{+}(x) & \text{in $\Sigma_{\lambda}$}, \\
 (u-u_{\lambda})^{-}(x) & \text{in $\mathbb{R}^{N} \setminus \Sigma_{\lambda}$},
\end{cases}
\end{equation}
where, {\it differently from before}, we have set
$$(u-u_{\lambda})^{-} := \min \{ u-u_{\lambda}, 0\},$$
which is {\it nonpositive}.
We claim that
\begin{equation}\label{LAB-LA5}
w_{\lambda}\in H^{1}(\mathbb{R}^{N}).
\end{equation} 
Indeed, we know that~$u\in H^1(\R^N)$ and thus~$u-u_\lambda\in H^1(\R^N)$.
Accordingly, we have that
(see e.g. the Chain Rule on page~296 of~\cite{LeoniBook})
\begin{equation}\label{LAB-LA1}
(u-u_\lambda)^+\in H^1(\R^N).
\end{equation}
Moreover, $u\in C(\R^N)$, and consequently
\begin{equation}\label{LAB-LA2}
(u-u_\lambda)^+\in C(\R^N).
\end{equation}
In addition, $u=u_\lambda$ along~$\partial\Sigma_\lambda$.
{F}rom this fact, \eqref{LAB-LA1}
and~\eqref{LAB-LA2}, we obtain that
\begin{equation}\label{LAB-LA3}
(u-u_\lambda)^+\chi_{\Sigma_\lambda}\in H^1_0(\Sigma_\lambda)
\subset H^1(\R^N),
\end{equation}
see e.g.~\cite[Theorem 9.17]{BrezisBook}.
Similarly,
\begin{equation}\label{LAB-LA4}
(u-u_\lambda)^-\chi_{\R^N\setminus
\Sigma_\lambda}\in H^1(\R^N).
\end{equation}
We also observe that
$$w_\lambda=(u-u_\lambda)^+\chi_{\Sigma_\lambda}+
(u-u_\lambda)^-\chi_{\R^N\setminus
\Sigma_\lambda}.$$
{F}rom this, \eqref{LAB-LA3}
and~\eqref{LAB-LA4},
we obtain~\eqref{LAB-LA5}, as desired.

Furthermore, we claim that
\begin{equation}\label{HY703}
{\mbox{$w_\lambda\equiv0$ in~$\R^N\setminus(\Omega_\lambda\cup Q_\lambda(\Omega_\lambda))
\subset\R^N\setminus\Omega$.}}\end{equation}
Indeed, if~$x\in\Sigma_\lambda\setminus\Omega_\lambda$,
then~$w_\lambda(x)=(0-u_\lambda(x))^+=0$.
If instead~$x\in Q_\lambda(\Sigma_\lambda\setminus\Omega_\lambda)$,
then~$\bar x\in \Sigma_\lambda\setminus\Omega_\lambda$
and accordingly
$$0=w_\lambda(\bar x)=(u(\bar x)-u_\lambda(\bar x))^+=
(u_\lambda(x)-u(x))^+.$$ 
This gives that~$u_\lambda(x)\le u(x)$, and therefore~$w_\lambda(x)=
(u(x)-u_\lambda(x))^-=0$.

{F}rom these observations, we obtain~\eqref{HY703}.
Then, \eqref{LAB-LA5} and~\eqref{HY703} give that we can take~$w_{\lambda}$
as an admissibile test function in~\eqref{ADMI}. In this way, we obtain 
\begin{equation}\label{eq:di_u}
B(u, w_\lambda)=	\int_{\mathbb{R}^{N}}f(u(x))w_{\lambda}(x) \, dx.
\end{equation}
Similarly, 
	\begin{equation}\label{eq:di_ulambda}
	B(u_\lambda,w_\lambda)=
\int_{\mathbb{R}^{N}}f(u_{\lambda}(x))w_{\lambda}(x) \, dx.
\end{equation}
Subtracting \eqref{eq:di_ulambda} to \eqref{eq:di_u},
and recalling~\eqref{BILI}, we get
\begin{equation}\begin{split} \label{od3gfhrj}
& \int_{\mathbb{R}^{N}} \scal{\nabla (u-u_{\lambda})}
{\nabla w_{\lambda}} \,dx \\
 & \quad\quad +\displaystyle \iint_{\mathbb{R}^{2N}}
 \dfrac{((u(x)-u_{\lambda}(x))-(u(y)-u_{\lambda}(y)))(w_{\lambda}(x)-w_{\lambda}(y))}
 {|x-y|^{N+2s}}\,\,dx\,dy\\
  & = \int_{\mathbb{R}^{N}}\big(f(u(x))-f(u_{\lambda}(x))\big)w_{\lambda}(x)\, dx.
\end{split}
\end{equation}
Now, we use formula (3.9) in \cite{MPS}, which gives that
\begin{align*}
 &\iint_{\mathbb{R}^{2N}} 
 \dfrac{\big((u(x)-u_{\lambda}(x))-(u(y)-u_{\lambda}(y))\big)
 (w_{\lambda}(x)-w_{\lambda}(y))}{|x-y|^{N+2s}}\, \,dx\,dy \\
 &\qquad \geq  
 \iint_{\mathbb{R}^{2N}} \dfrac{|w_{\lambda}(x)-w_{\lambda}(y)|^2}{|x-y|^{N+2s}}\,dx\,dy \geq 0.
 \end{align*}
Using this information into~\eqref{od3gfhrj}, and
recalling~\eqref{HY703}, we obtain that
\begin{equation}\label{djietu86yghjfv30o22}
\begin{split}
& \int_{\mathbb{R}^{N}} \scal{\nabla (u-u_{\lambda})}
 {\nabla w_{\lambda}} \,dx
  \le \int_{\mathbb{R}^{N}}\big(f(u(x))-f(u_{\lambda}(x))\big)w_{\lambda}(x)\,dx \\
 &\qquad = \int_{\mathbb{R}^{N}}\frac{f(u(x))-f(u_{\lambda}(x))}
 {u(x)-u_\lambda(x)}\,\big(u(x)-u_\lambda(x)\big)w_{\lambda}(x)\,dx \\
 & \qquad = \int_{\mathbb{R}^{N}}\frac{f(u(x))-f(u_{\lambda}(x))}
  {u(x)-u_\lambda(x)}w_{\lambda}^2(x)\,dx \\ 
 & \qquad = \int_{\Omega_{\lambda} \cup Q(\Omega_{\lambda})}
 \frac{f(u(x))-f(u_{\lambda}(x))}{u(x)-u_\lambda(x)}w_{\lambda}^2(x)\,dx.
 \end{split}
 \end{equation}
We also notice that, thanks to~\eqref{HY703},
\begin{equation*}
\int_{\mathbb{R}^{N}} \scal{\nabla (u-u_{\lambda})}
 {\nabla w_{\lambda}}\,dx 
 = \int_{\mathbb{R}^{N}}|\nabla w_{\lambda}|^2\,dx 
	= \int_{\Omega_{\lambda} \cup Q(\Omega_{\lambda})}|\nabla w_{\lambda}|^2\,dx.
\end{equation*}
{F}rom this and~\eqref{djietu86yghjfv30o22}, we deduce that
\begin{equation} \label{od3gfhrjBIS}
\begin{split}
\int_{\Omega_{\lambda} \cup Q(\Omega_{\lambda})}|\nabla w_{\lambda}|^2\,dx
& \leq  
\int_{\Omega_{\lambda} \cup Q(\Omega_{\lambda})}
\frac{f(u(x))-f(u_{\lambda}(x))}{u(x)-u_\lambda(x)}w_{\lambda}^2(x)\,dx \\[0.1cm]
& \leq 
C \int_{\Omega_{\lambda} \cup Q(\Omega_{\lambda})}|w_{\lambda}|^2 \,dx,
\end{split}
\end{equation}
for some constant~$C>0$, depending on~$f$ and~$\|u\|_{L^\infty(\Omega)}$.

Now, using Lemma~2.10 in~\cite{BVV}, we obtain that
\begin{equation}\label{Grad}
\int_{\Omega_{\lambda} \cup Q(\Omega_{\lambda})}|\nabla w_{\lambda}|^2\,dx
	\leq C|\Omega_{\lambda} \cup Q(\Omega_{\lambda})|^{1/N}
	\int_{\Omega_{\lambda} \cup Q(\Omega_{\lambda})}|
\nabla w_{\lambda}|^2\,dx,
\end{equation}
uo to renaming~$C$, which possibly depends also on~$N$.
As a consequence, if $\lambda$ is sufficiently close to~$-1$, we see that
\begin{equation*}
 C|\Omega_{\lambda} \cup Q(\Omega_{\lambda})|^{1/N} < \dfrac{1}{2},
\end{equation*}
which, combined with~\eqref{Grad}, gives that
$$\int_{\Omega_{\lambda} \cup Q(\Omega_{\lambda})}
|\nabla w_{\lambda}|^2 \,dx= 0,$$
provided that $\lambda$ is sufficiently close to $-1$. 
{F}rom this and the Poincar\'{e} inequality we get
that~$w_{\lambda} \equiv 0$ in $\Omega_{\lambda} 
\cup Q_\lambda(\Omega_{\lambda})$ if~$\lambda$
is sufficiently close to~$-1$, which, recalling~\eqref{eq:wlambda},
implies that
\begin{equation}\label{spwpr4oytyo6789}
u\le u_\lambda \quad \text{in $\Omega_\lambda$},
\end{equation}
if~$\lambda$ is sufficiently close to~$-1$.
As a matter of fact, formula~\eqref{HY703} also gives
\begin{equation}\label{spwpr4oytyo6789BIS}
u\le u_\lambda \quad \text{in $\Sigma_\lambda\setminus\Omega_\lambda$}.
\end{equation}
Now, we define the set 
$$\Lambda_{0} := \big\{\lambda \in (-1,0):\,\,\text{$u \leq u_t \textrm{ in } \Omega_{t}$ 
for every $t \in (-1,\lambda]$}\big\}.$$
In light of~\eqref{spwpr4oytyo6789}, the following quantity is well defined:
\begin{equation}\label{CONTR}
\overline{\lambda} := \sup \Lambda_{0}.\end{equation}
The goal is now to prove that
\begin{equation}\label{CONTRBIS}
\overline{\lambda}= 0.
\end{equation}
For this, we recall the definition of~$v_\lambda$ in~\eqref{degjidj},
and we observe that,
since $u$ is continuous in $\Omega$,
$$v_{\overline{\lambda}}\ge0, 
\quad \textrm{in } \Omega_{\overline{\lambda}}.$$
This and~\eqref{spwpr4oytyo6789BIS} imply that
$$v_{\overline{\lambda}}\ge0, 
\quad \textrm{in } \Sigma_{\overline{\lambda}}.$$
As a consequence, by Lemma~\ref{lema:weak} and
Proposition~\ref{StrongMax} (applied here with~$H:=\Sigma_{\overline{\lambda}}$,
$U:=\Omega_{\overline{\lambda}}$ and~$v:=v_{\overline{\lambda}}$),
we have that
$$v_{\overline{\lambda}}>0, \quad \textrm{in } \Omega_{\overline{\lambda}}.$$
Now, we consider a compact
set~$K \subset \Omega_{\overline{\lambda}}$ (to be chosen later on),
and we notice
that for~$\overline{\tau}>0$ small enough, we have that
\begin{equation}\label{usoqui}
\text{$v_{\overline{\lambda}+\tau}>0$\,\,in $K$}\qquad
(\text{for all $\tau \in (0, \overline{\tau})$}).
\end{equation}
Now, for every fixed $\tau \in (0, \overline{\tau})$, we consider
the function~$w_{\overline{\lambda}+\tau}$ defined as
in~\eqref{eq:wlambda} (with~$\lambda:=\overline{\lambda}+\tau$).
We notice that, thanks to~\eqref{LAB-LA5} and~\eqref{HY703},
we can take~$w_{\overline{\lambda}+\tau}$
as an admissibile test function in~\eqref{ADMI}, obtaining that 
\begin{equation*}
\begin{split}
& B(u, w_{\overline{\lambda}+\tau})=	\int_{\mathbb{R}^{N}}
f(u(x))w_{{\overline{\lambda}+\tau}}(x) \, dx\qquad\text{and} \\
& \quad B(u_{\overline{\lambda}+\tau},w_{\overline{\lambda}+\tau})=
\int_{\mathbb{R}^{N}}f(u_{\overline{\lambda}+\tau}(x))
w_{\overline{\lambda}+\tau}(x) \, dx.
\end{split}
\end{equation*}
{F}rom here, we repeat the same argument
in~\eqref{od3gfhrj}--\eqref{od3gfhrjBIS} to find that
$$\int_{\Omega_{\overline{\lambda}+\tau} \cup Q(\Omega_{\overline{\lambda}+\tau})}
|\nabla w_{\overline{\lambda}+\tau}|^2\,dx
\le
C \int_{\Omega_{\overline{\lambda}+\tau} 
\cup Q(\Omega_{\overline{\lambda}+\tau})}|w_{\overline{\lambda}+\tau}|^2 \,dx,
$$
for some constant~$C>0$,
depending on~$f$ and~$\|u\|_{L^\infty(\Omega)}$.

{F}rom this, re\-cal\-ling~\eqref{usoqui}, we obtain that
\begin{equation*}
\int_{\Omega_{\overline{\lambda}+\tau} \cup Q(\Omega_{\overline{\lambda}+\tau})}
|\nabla w_{\overline{\lambda}+\tau}|^2\,dx
\le C\int_{(\Omega_{\overline{\lambda}+\tau}\setminus K) 
\cup Q(\Omega_{\overline{\lambda}+\tau}\setminus K)}|w_{\overline{\lambda}+\tau}|^2\,dx.	
\end{equation*}
Hence, 
making again use of Lemma~2.10 in~\cite{BVV}, we get
\begin{equation}\begin{split}\label{det4458645gdf}
&	\int_{\Omega_{\overline{\lambda}+\tau} 
\cup Q(\Omega_{\overline{\lambda}+\tau})}|\nabla w_{
\overline{\lambda}+\tau}|^2 \,dx
	\\&\qquad \leq C\,
|(\Omega_{\overline{\lambda}+\tau}\setminus K)
\cup Q(\Omega_{\overline{\lambda}+\tau}\setminus K)|^{1/N}
\int_{(\Omega_{\overline{\lambda}+\tau}\setminus K) \cup
 Q(\Omega_{\overline{\lambda}+\tau}\setminus K)}|
\nabla v_{\overline{\lambda}+\tau}|^2\,dx,
\end{split}	\end{equation}
up to relabeling~$C>0$ (which may also depend on~$N$).
Now we choose the compact $K$ big enough and the number
$\overline{\tau}$
small enough such that
$$C\,
|(\Omega_{\overline{\lambda}+\tau}\setminus K) 
\cup Q(\Omega_{\overline{\lambda}+\tau}\setminus K)|^{1/N} < 1.$$
Using this information into~\eqref{det4458645gdf}, we conclude that 
\begin{equation*}
 \int_{\Omega_{\overline{\lambda}+\tau} \cup Q(\Omega_{\overline{\lambda}+\tau})}|
 \nabla w_{\overline{\lambda}+\tau}|^2\,dx = 0.
	\end{equation*}
{F}rom this and the Poincar\'e inequality, we find
that~ $w_{\overline{\lambda}+\tau}\equiv0$ in~$\Omega_{
\overline{\lambda}+\tau}$, hence
	$$u \leq u_{\overline{\lambda}+\tau} \quad \text{in $\Omega_{\overline{\lambda}+\tau}$},$$
for every $\tau \in (0, \overline{\tau})$,
provided $\overline{\tau}>0$ is small enough. This yields a contradiction
with~\eqref{CONTR}, from which we conclude that~\eqref{CONTRBIS}
holds true, as desired. 

In particular, from~\eqref{CONTRBIS} we see that, for all~$\lambda\in(-1,0)$
and all~$x\in\Omega_\lambda$,
\begin{equation}\label{detr34y54yu45u4u}
u(x)\le u_\lambda(x)=u(2\lambda-x_1, x_2,\cdots,x_N).\end{equation}
Consequently, 
\begin{equation}\label{s3r34t43yh}
u(x)\le u(-x_1, x_2,\cdots,x_N),\end{equation}
for all~$x\in\Omega\cap\{x_1<0\}$.
In the same way, sliding the moving plane from right to left, one sees that,
for all~$x\in\Omega\cap\{x_1>0\}$, one has
$$ u(x)\le u(-x_1, x_2,\cdots,x_N).$$
This implies that
$$ u(-x_1, x_2,\cdots,x_N)\le u(x),$$
for all~$x\in\Omega\cap\{x_1<0\}$.
{F}rom this and~\eqref{s3r34t43yh}, we conclude that
$$ u(x)=u(-x_1, x_2,\cdots,x_N),$$
for all~$x\in\Omega$, which says that~$u$ is symmetric with respect
to~$\{x_1=0\}$.

Furthermore, from~\eqref{detr34y54yu45u4u} it plainly follows
that~$u$ is increasing in the~$x_1$-di\-rec\-tion in~$\Omega\cap\{x_1<0\}$.
The proof of Theorem~\ref{Simmetria} is thereby complete.
\end{proof}

\section{One-dimensional symmetry and proof of Theorem~\ref{thm.Gibbons}} \label{sec.main}

In this section we provide the proof of Theorem~\ref{thm.Gibbons}.
For this, we indicate the points $x\in\RN$ by 
 $$\text{$(y,t)$, with $y\in\R^{N-1}$ and
 $t\in\R$}.$$
 Moreover, we consider the functional space
 \begin{equation}\label{defix}
  \bbX  := C^3(\RN)\cap W^{4,\infty}(\RN).
  \end{equation}
  We point out that, if $u\in\bbX$, it is possible
  to compute $\LL u$ in the classical sense, that is, $\LL u(x)$
  is well-defined for all $x\in\RN$.

We shall derive Theorem~\ref{thm.Gibbons} from the abstract
  approach developed in~\cite{FarVal}.
To this end, we check that
the assumptions introduced in~\cite{FarVal} are satisfied in our setting.
We list these assumptions here for the convenience of the reader:

  \begin{description}
   \item[(H1)] if $\varphi\in\bbX$
    satisfies $\LL\varphi = f(\varphi)$ in $\RN$, then there exists an 
    operator $\tilde{\LL}$, ac\-ting on a suitable
    space of functions $\tilde{\bbX}\subseteq C(\RN)$ which is
translation-invariant\footnote{A (non-void) set $V\subseteq C(\RN)$ is
\emph{translation-invariant}
    if, for every function $\varphi\in V$ and every point $y\in\RN$, the `tran\-sla\-ted' function
    $x\mapsto \varphi(x+y)$
    belongs to $V$.},
    such that $\de_\nu \varphi\in \tilde{\bbX}$ for any unit vector
    $\nu\in\RN$ and
    $$\tilde{\LL}(\de_\nu\varphi) = f'(\varphi)\,\de_{\nu}\varphi\qquad\text{on $\RN$};$$
    
   \item[(H2)] if $\varphi\in\bbX$ is a solution of \eqref{eq.GibbonsPb},
   if
   $\{z_k\}_{k = 1}^\infty$ is an arbitrary sequence of points
   in $\RN$
   (possibly unbounded) and if
   $$\varphi_k := \varphi(\cdot+z_k) \qquad {\mbox{ for any }} k\in\N,$$
   then there exists a function $\varphi_0\in\bbX$ such that, up to a sub-sequence,
   \begin{eqnarray*}&&
     \lim_{k\to\infty}\varphi_k(x)  = \varphi_0(x), \\
&&    \lim_{k\to\infty}\nabla \varphi_k(x) = \nabla\varphi_0(x) 
\\ \text{and} \quad &&
  \lim_{k\to\infty}\LL\varphi_k(x) = \varphi_0(x),
   \end{eqnarray*}
for all~$x\in\RN$;
   
   \item[(H3)] if $w\in \tilde{\bbX}$ satisfies $\tilde{\LL} w + c(x)w=0$ in $\RN$,
   with 
   $$\text{$w(y,t)\geq 0$ if $|t|\leq M$ and $c(y,t)\geq \kappa$ if $|t|\geq M$}$$
   for some constants $M,\,\kappa > 0$, then 
   $$w(x)\geq 0 \qquad\text{for all $x\in\RN$};$$
   
   \item[(H4)] if $\varphi\in\bbX$ and 
   if $w\in\tilde{\bbX}$ satisfies $\tilde{\LL}w = f'(\varphi)w$ in $\RN$, then
   $$
   \begin{cases}
   \text{$w\geq 0$ in $\RN$}, \\
   w(0) = 0,
   \end{cases}\,\,\Longrightarrow\,\,\text{$w\equiv 0$ on $\RN$};$$
   
   \item[(H5)] given $\mu_-<\mu_+\in\R$, if $U\subseteq \RN$ 
   is an open set contained in 
   $$\mathcal{S} := \{x = (y,t)\in\RN:\,\text{$t\leq\mu_-$ or $t\geq \mu_+$}\}$$
   and if $v \in\bbX$ 
   satisfies $\LL v +c(x)v=0$ in $\RN$, with 
   $$\text{$v(x) \geq $ in $\RN\setminus U$ and $c(x)\geq \kappa$ on $U$}$$
   for some constant $\kappa > 0$, then
	$$v(x)\geq 0\qquad\text{for all $x\in\RN$};$$
   
   \item[(H6)] if $\varphi\in\bbX$ and 
   if $v\in\bbX$ satisfies ${\LL}v = f(v+\varphi)-f(v)$ in $\RN$, then
   $$
   \begin{cases}
   \text{$v\geq 0$ in $\RN$}, \\
   v(0) = 0,
   \end{cases}\,\,\Longrightarrow\,\,\text{$v\equiv 0$ on $\RN$}.$$
  \end{description}
The next lemmata establish
the validity of (H1)---(H6) in our setting.
 \begin{lemma}[Validity of~(H1)] \label{lem.H1}
  For every $\varphi\in\bbX$ and every unit vector $\nu\in\RN$, one has
  \begin{equation} \label{eq.LLcommutedenu}
   \LL(\de_\nu \varphi) = \de_\nu\big(\LL\varphi).
   \end{equation}
  In particular, assumption 
  \emph{(H1)} is fulfilled with the choices
  \begin{equation}\label{deftildel}
  \tilde{\LL} := \LL
  \end{equation} and
  \begin{equation}\label{eq.deftildeX}
   \tilde{\bbX} := C^2(\RN)\cap W^{3,\infty}(\RN).
    \end{equation}
  \end{lemma}

  \begin{proof}
    First of all, if~$\bbX$ is as in~\eqref{defix} and~$\tilde{\bbX}$
is as in~\eqref{eq.deftildeX}, we obviously have that,
for every $\varphi\in\bbX$ and every unit vector $\nu\in\RN$,
   $$\de_\nu\varphi\in\tilde{\bbX}\qquad\text{and}\qquad
   -\Delta(\de_\nu\varphi) = \de_\nu(-\Delta \varphi).
   $$
   Moreover, since $\bbX\subseteq W^{3,\infty}(\RN)$, we can
   use formula~(4.1) in~\cite{FarVal}, obtaining that
   $$(-\Delta)^s(\de_\nu \varphi) = \de_\nu\big((-\Delta)^s\varphi).$$
   Gathering together these facts, we obtain~\eqref{eq.LLcommutedenu},
  as desired.
  As a result, with the choices in~\eqref{deftildel}
  and~\eqref{eq.deftildeX}, assumption~(H1) is obviously satisfied.
  \end{proof}
   We point out that the space $\tilde{\bbX}\supseteq\bbX$ is `good' 
   for dealing with~$\LL$. Indeed, since any function 
   $u\in \tilde{\bbX}$
  has bounded derivatives up to second order, 
  we can compute~$\LL u$ pointwise in $\RN$.
  \begin{lemma}[Validity of~(H2)] \label{lem.H2}
   Let  $\tilde{\bbX} $ be as in~\eqref{eq.deftildeX}.
   Let $\varphi\in \tilde{\bbX}$ and~$\{z_k\}_{k = 1}^\infty$
   be a sequence of points in~$\RN$
   \emph{(}possibly unbounded\emph{)}. Let also
   \begin{equation}\label{swryerhrh}
    \varphi_k := \varphi(\cdot+z_k) \quad \text{for any $k\in\N$}.
    \end{equation}
   Then, there exists a function $\varphi_0\in\bbX$ such that, up to a sub-sequence,
   \begin{align}
     &\lim_{k\to\infty}\varphi_k(x) = \varphi_0(x), \label{punto1}\\
     & \lim_{k\to\infty}\nabla \varphi_k(x) = \nabla\varphi_0(x)\label{punto2}\\
     \text{and} \quad & \lim_{k\to\infty}\LL\varphi_k(x) = \LL\varphi_0(x),\label{punto3}
   \end{align}
   for all~$x\in\RN$. In particular, assumption \emph{(H2)} is fulfilled.
  \end{lemma}
  \begin{proof}
  We observe that, since $\varphi\in\tilde{\bbX}$,
  the sequences
   $$\{D^\alpha\varphi_k\}_{k = 1}^\infty$$ 
   are equi-continuous
   and equi-bounded on~$\RN$, for every multi-index
   $\alpha\in\N^N$
   sa\-ti\-sfy\-ing $0\leq|\alpha|\leq 2$.
   As a consequence, Arzel\`a-Ascoli's Theorem ensures the e\-xi\-sten\-ce
   of some function~$\varphi_0\in\bbX$ such that (up to a sub-sequence)
   \begin{equation} \label{eq.limitDalpha}
    \text{$
	 \dsy\lim_{k\to\infty}
	 D^\alpha\varphi_k = D^\alpha\varphi_0$
	 \quad locally uniformly in $\RN$},
   \end{equation}
   for every $\alpha\in\N^N$ with $|\alpha|\leq 2$.
   Hence, \eqref{punto1} and~\eqref{punto2} plainly follows from~\eqref{eq.limitDalpha}.
   We also deduce from~\eqref{eq.limitDalpha} that
   \begin{equation} \label{eq.limitDelta}
    \lim_{k\to\infty}\Delta \varphi_k(x) = 
   \Delta\varphi_0(x)\quad\text{locally uniformly in $\RN$}.
   \end{equation}
  We now claim that
   \begin{equation} \label{eq.proveDeltasconverges}
    \lim_{k\to\infty}(-\Delta)^s\varphi_k(x) = (-\Delta)^s \varphi_0(x)
 \qquad\text{for every $x\in\RN$}.
   \end{equation}
  To prove it, for any~$x\in\RN$ and for any~$k\in\N$, we set 
   $$\mathcal{I}_k(z) := 
   \frac{\varphi_k(x+z)-\varphi_k(x-z)-2\varphi_k(x)}{|z|^{N+2s}}
   \qquad {\mbox{ for any }}z\neq 0.$$
On account of \eqref{eq.limitDalpha}, we have that
 \begin{equation}\label{ergerjtrjyi7i76io76}
   \lim_{k\to\infty}\mathcal{I}_k(z) = 
   \frac{\varphi_0(x+z)-\varphi_0(x-z)-2\varphi_0(x)}{|z|^{N+2s}} \qquad\text{
   for all $z\neq 0$}.
   \end{equation}
 Moreover, recalling the definition of $\varphi_k$ in~\eqref{swryerhrh},
 we see that, for every $z\neq 0$,
 \begin{equation}\label{ergerjtrjyi7i76io7622}
\begin{split}
    |\mathcal{I}_k(z)|  =\;&
  \frac{\big|\varphi_k(x+z)+\varphi_k(x-z)|-2\varphi_k(x)\big|}{|z|^{N+2s}}\\  
  & \leq 
    \max_{|\alpha|= 2}\|D^\alpha \varphi_k\|_{L^\infty(\RN)}\,
    \frac{1}{|z|^{N+2s-2}}\,\chi_{\{0<|z|\leq 1\}}
    \\
    & \qquad + 4\|\varphi_k\|_{L^\infty(\RN)}\,\frac{1}{|z|^{N+2s}}\,
    \chi_{\{|z| > 1\}} \\
 & = 
    \max_{|\alpha|= 2}\|D^\alpha \varphi\|_{L^\infty(\RN)}\,
    \frac{1}{|z|^{N+2s-2}}\,\chi_{\{0<|z|\leq 1\}}
    \\
     &\qquad + 4\|\varphi\|_{L^\infty(\RN)}\,\frac{1}{|z|^{N+2s}}\,\chi_{\{|z| > 1\}}.
   \end{split}
   \end{equation}
 Now, since~$\varphi\in\tilde{\bbX}$, we have that
 \begin{align*} 
 g(z) & :=
    \max_{|\alpha|= 2}\|D^\alpha \varphi\|_{L^\infty(\RN)}\,
    \frac{1}{|z|^{N+2s-2}}\,\chi_{\{0<|z|\leq 1\}} \\
 &\qquad + 4\|\varphi\|_{L^\infty(\RN)}\,\frac{1}{|z|^{N+2s}}\,
  \chi_{\{|z| > 1\}}\in L^1(\R^N).
  \end{align*}
 {F}rom this, \eqref{ergerjtrjyi7i76io76} and~\eqref{ergerjtrjyi7i76io7622}
 we deduce that we can apply the Dominated Con\-ver\-gen\-ce Theorem
 to conclude that, for any~$x\in\R^N$,
 \begin{align*}
 & \lim_{k\to\infty}\int_{\R^N}
 \frac{\varphi_k(x+z)-\varphi_k(x-z)-2\varphi_k(x)}{|z|^{N+2s}}\,dz \\
 &\qquad=\int_{\R^N}
 \frac{\varphi_0(x+z)-\varphi_0(x-z)-2\varphi_0(x)}{|z|^{N+2s}}\,dz.
 \end{align*}
 This proves~\eqref{eq.proveDeltasconverges}.
 {F}rom~\eqref{eq.limitDelta} and~\eqref{eq.proveDeltasconverges}, 
 recalling~\eqref{OPER}, we obtain~\eqref{punto3}.
 Finally, since~$\bbX\subset\tilde{\bbX}$, we deduce that
 assumption~{(H2)} is fulfilled,
 thus completing the proof of Lemma~\ref{lem.H2}.
 \end{proof}
 \begin{lemma}[Validity of~(H3) and~(H5)] \label{lem.H35}
 Let~$\tilde{\bbX}$ be as in~\eqref{eq.deftildeX}.
 Let $w \in\tilde{\bbX}$ satisfy 
 \begin{equation}\label{sw45009876qwagdfhtfijgsddwsrb}
  \LL w + c(x)w=0\quad \text{in $\RN$},
 \end{equation} with 
   \begin{equation}\label{swr3t4365749uyti}
   \text{$w(x) \geq0 $ in $\RN\setminus U$ and $c(x)\geq \kappa$ on $U$}
  \end{equation}
  for some open set $U\subseteq\RN$ and some constant $\kappa > 0$. Then
  \begin{equation}\label{THIS}
  w(x)\geq 0\qquad\text{for all $x\in\RN$}.
  \end{equation}
 In particular, assumptions~\emph{(H3)} and~\emph{(H5)}
 are fulfilled with the choice in~\eqref{deftildel}.
  \end{lemma}
  \begin{proof}
   Arguing by contradiction, we suppose that $m := \inf_{\RN}w < 0$, and we choose
   a sequence of points $\{z_k\}_{k = 1}^\infty$ in $\RN$ satisfying
   \begin{equation} \label{eq.limwzkmU} 
   \lim_{k\to\infty} w(z_k) = m.
   \end{equation}
   Since $m < 0$, it is not restrictive to assume that
   \begin{equation} \label{eq.wzknegativeU}
    w(z_k) \leq \frac{m}{2} < 0 \qquad\text{for all $k\in\N$}.
    \end{equation}
   As a consequence, also in light of~\eqref{swr3t4365749uyti}, for
   every $k\in\N$ we have
   \begin{equation} \label{eq.czkpositiveU}
    z_k\in U \qquad\text{and}\qquad c(z_k)\geq \kappa > 0.
    \end{equation}
  Now, thanks to~\eqref{sw45009876qwagdfhtfijgsddwsrb},
   from \eqref{eq.wzknegativeU} and \eqref{eq.czkpositiveU} we deduce that
   $$\LL w(z_k) =- c(z_k)w(z_k) \geq -\frac{m\,\kappa}{2} > 0 ,
  \qquad\text{for all $k\in\N$}.$$
   In particular, setting $w_k:= w(\cdot+z_k)$, we obtain
   \begin{equation} \label{eq.topasslimitU}
    \LL w_k(0) \geq -\frac{m\,\kappa}{2}>0,\qquad\text{for all $k\in\N$}.
   \end{equation}
   On the other hand, since $w\in\tilde{\bbX}$, from Lemma~\ref{lem.H2}
   we infer the existence of some
   function $w_0\in \tilde{\bbX}$ such that (up to a sub-sequence)
  \begin{equation}\label{swgeryreyhjlio403683uigfds}
   \lim_{k\to\infty}w_k(x) = w_0(x)\qquad\text{and}\qquad
   \lim_{k\to\infty}\LL w_k(x) = \LL w_0(x),
   \end{equation}
   for every fixed $x\in\RN$. By taking the limit
   as $k\to\infty$ in \eqref{eq.topasslimitU}, we then get
   \begin{equation} \label{eq.tocontradictU}
    \LL w_0(0) \geq -\frac{m\,\kappa}{2} > 0.
   \end{equation}
  Now, we observe that,
   on account of \eqref{eq.limwzkmU} and~\eqref{swgeryreyhjlio403683uigfds}, one has
   \begin{align*}
    w_0(0) = \lim_{k\to\infty}w_k(0) = \lim_{k\to\infty}
    w(z_k) = m = \inf_{\RN}w \leq w(x+z_k) = w_k(x),
   \end{align*}
   for every $x\in\RN$ and every~$k\in\N$. As a consequence,
   $$w_0(0) \leq w_0(x)\qquad\text{for every $x\in\RN$},$$
   and thus $x = 0$ is a minimum point for $w_0$ in $\RN$. In particular,
   $$\Delta w_0(0)\geq 0 \qquad \text{and} \quad -(-\Delta)^s w_0(0)
   = \mathrm{P.V.}\int_{\RN}\frac{w_0(x)-w_0(0)}{|x|^{N+2s}}\,d x \geq 0.$$
  Therefore, recalling~\eqref{OPER},
  this implies that~$\LL w_0(0)\leq 0$, 
   which is is in con\-tra\-dic\-tion with 
   \eqref{eq.tocontradictU}. This completes the proof of~\eqref{THIS}.

  We point out that, with the choice in~\eqref{deftildel}, from the first part
of Lemma \ref{lem.H35} we obtain the validity of assumption~(H3). 
Indeed, for this, it is enough to apply
the first part
of Lemma \ref{lem.H35} with
 $$U := \big\{x = (y,t)\in\RN \,{\mbox{ s.t. }}\,|t| \ge M\big\} ,$$
for some $M > 0$.
Furthermore, from the first part
of Lemma \ref{lem.H35} we also obtain the validity of assumption~(H5),
by simply
observing that~$\bbX\subset\tilde{\bbX}$.
\end{proof}
  \begin{lemma}[Validity of~(H4) and~(H6)] \label{lem.H46}
Let~$\tilde{\bbX}$ be as in~\eqref{eq.deftildeX}.
   Let $c:\RN\times\R\to\R$ be any function satisfying
   \begin{equation} \label{eq.cvanishzero}
    c(x,0) = 0\qquad\text{for every $x\in\RN$}.
   \end{equation}
Let~$w\in\tilde{\bbX}$ satisfy 
\begin{equation}\label{weiwtyegsdjk}
\LL w + c(x,w)=0\quad {\mbox{ in }}\RN.
\end{equation} Then
\begin{equation}\label{assfyeryuer}
   \begin{cases}
   \text{$w\geq 0$ in $\RN$}, \\
   w(0) = 0,
   \end{cases}\,\,\Longrightarrow\,\,\text{$w\equiv 0$ on $\RN$}.
\end{equation}
 In particular, assumptions~\emph{(H4)} and~\emph{(H6)}
 are fulfilled with the choices in~\eqref{deftildel}
 and~\eqref{eq.deftildeX}.
  \end{lemma}
  \begin{proof}
 We observe that, thanks to the assumptions in~\eqref{assfyeryuer},
 $x = 0$ is a minimum point for~$w$ in~$\RN$.
 As a consequence, we have that
 \begin{equation}\label{we098r7654}
 \Delta w(0)\geq 0 \qquad\text{and}\quad 
   -(-\Delta)^s w(0) =
   \mathrm{P.V.}\int_{\RN}\frac{w(x)}{|x|^{N+2s}}\,d x \geq 0.
\end{equation}
 On the other hand, by~\eqref{eq.cvanishzero} and~\eqref{weiwtyegsdjk},
 and recalling also that~$w(0) = 0$, we get
 $$
    0 = c(0,0) = c(0,w(0)) 
    = -\LL w(0) = \Delta w(0)-(-\Delta)^s w(0) \geq -(-\Delta)^s w(0).
 $$
   Gathering together this and~\eqref{we098r7654}, we conclude that
   $$0 = -(-\Delta)^s w(0) = \mathrm{P.V.}\int_{\RN}\frac{w(x)}{|x|^{N+2s}}\,dx.$$
  Since $w\geq 0$ in $\RN$, we deduce that~$w\equiv 0$
  on the whole of~$\RN$, which completes the proof of the 
  claim in~\eqref{assfyeryuer}. \medskip

  Now, we check the validity of assumption~(H4). For this, recalling
  \eqref{deftildel} and~\eqref{eq.deftildeX},
  we take~$\varphi\in\bbX$  and we define
  $$c(x,w) := -f'(\varphi(x))\,w.$$
  We observe that~$c$ satisfies~\eqref{eq.cvanishzero}.
  Hence, we can apply the first part of Lemma~\ref{lem.H46}
  to obtain that~(H4) is satisfied.
  Finally, in order to show the validity of as\-sump\-tion~(H6), given~$\varphi\in\bbX$, 
  we define
   $$c(x,w) := f(\varphi(x)+w)-f(\varphi(x)).$$
  This function satisfies \eqref{eq.cvanishzero}. As a consequence of this
  and of the inclusion $\bbX\subseteq\tilde{\bbX}$, we deduce~(H6)
  from the first part of Lemma~\ref{lem.H46}.
  \end{proof}
 Thanks to these statements, we can now prove Theorem \ref{thm.Gibbons}:
  
\begin{proof}[Proof of Theorem \ref{thm.Gibbons}]
  On account of Lemmata \ref{lem.H1}, \ref{lem.H2}, \ref{lem.H35} and
  \ref{lem.H46}, we know that the assumptions
  in~(H1)---(H6) are fulfilled
  in the setting of Theorem \ref{thm.Gibbons}.
 Moreover, since $u\in\bbX$, we have that
  $$\text{$\|u\|_{C^{1,\beta}(\RN)}$ is finite for all $\beta\in(0,1)$}.$$
 {F}rom these considerations and~\eqref{eq.assumptionf}, we have that
 the assumptions of Theorem~1.1 in~\cite{FarVal} are satisfied. Hence, 
 from Theorem~1.1 in~\cite{FarVal} we have that
 there exists some function $u_0:\R\to\R$ such that~\eqref{eq.onedim} holds true.
  \end{proof}   
	
 \vfill

\end{document}